\documentclass[a4paper,11pt]{article}
\usepackage{latexsym}
\usepackage{amsfonts}
\usepackage{amsthm}
\usepackage{amsmath}
\usepackage{amssymb}
\usepackage{amscd}
\usepackage[all,cmtip]{xy}
\usepackage{rotating}
\usepackage{scalefnt}
\usepackage[english]{babel}
\usepackage{makeidx}
\usepackage[latin1]{inputenc}
\usepackage{fancyhdr}
\usepackage{diagrams}
\newtheorem{prop}{Proposition} 
\newtheorem{lemma}{Lemma} 
\newtheorem{teo}{Theorem}

\theoremstyle{remark}
\newtheorem{definition}{Definition}
\newtheorem{notation}{Notation}
\newtheorem{remark}{Remark}
\newtheorem{example}{Example}
\begin{document}

\title{\textsc{Splitting in the $K$-theory localization sequence of number fields}}
\author{Luca Caputo\footnote{Partially supported by an IRCSET fellowship.}}
\date{\today}
\maketitle

\textsc{Abstract -} 
\begin{small}
Let $p$ be a rational prime and let $F$ be a number field. Then, for each $i\geq1$, there is an exact localization sequence 
$$0\longrightarrow K_{2i}(\mathcal{O}_F)_p\longrightarrow K_{2i}(F)_p \stackrel{\partial_{F,i}}{\longrightarrow} \bigoplus_{v\nmid\infty } K_{2i-1}(k_v)_p\longrightarrow 0$$
where $K_j(R)$ denotes the $j$-th Quillen's $K$-theory group of the ring $R$. If $p$ is odd or $F$ is nonexceptional, we find necessary and sufficient conditions for this exact sequence to split: these conditions involve coinvariants of twisted $p$-parts of the $p$-class groups of certain subfields of the fields $F(\mu_{p^n})$ for $n\in\mathbb{N}$. We also compare our conditions with the weaker condition $W\!K^{\acute{e}t}_{2i}(F)=0$ and give some example.
\end{small}   \\

\section{Introduction and notation}
Throughout the paper, $p$ will denote a rational prime. For an abelian group $A$, set 
$$\mathrm{Div}(A)=\textrm{maximal divisible subgroup of $A$}$$
$$\mathrm{div}(A)=\{a\in A\,|\,\forall\, n\in \mathbb{N}\,\, \exists\, a_n\in A : a=na_n\}$$ 
Then $\mathrm{div}(A)$ is a subgroup of $A$ which is commonly called the subgroup of (infinitely) divisible elements or the subgroup of elements of infinite height of $A$. We denote by $A_p$ the $p$-primary part of $A$ and, for $n\in \mathbb{N}$, by $A[p^n]$ the subgroup of elements of $A$ whose order divides $p^n$.\\  

If $R$ is a ring and $j$ is a natural number, $K_j(R)$ denotes the $j$-th Quillen's $K$-group of $R$. Let $F$ be a number field and $i$ be a positive integer. Thanks to Soulé's results (see \cite{We1}, Theorem 4.6), Quillen's long exact localization sequence splits into isomorphisms $K_{2i+1}(\mathcal{O}_F)\cong K_{2i+1}(F)$ and short exact sequences of the form  
\begin{equation}\label{es}
0\longrightarrow K_{2i}(\mathcal{O}_F)\longrightarrow K_{2i}(F) \stackrel{\partial_{F,i}}{\longrightarrow} \bigoplus_{v\nmid \infty} K_{2i-1}(k_v)\longrightarrow 0
\end{equation}  
where $\mathcal{O}_F$ is the ring of integers of $F$, $k_v$ is the residue field of $F$ at $v$ and the direct sum is taken over the finite places of $F$. We recall that, thanks to Quillen's and Borel's results, $K_{2i-1}(k_v)$ is cyclic of order $|k_v|^i-1$ and $K_{2i}(\mathcal{O}_F)$ is a finite group: in particular we always have $\mathrm{Div}(K_{2i}(F))=0$ (but in general $\mathrm{div}(K_{2i}(F))$ may be nontrivial).\\ 
One can asks for conditions for the exact sequence in (\ref{es}) to split. As a motivation for this question we quote the following three results:
\begin{itemize}
  \item if $F=\mathbb{Q}$, Tate showed that the exact sequence (\ref{es}) splits (see \cite{Mi2}, Theorem 11.6);
	\item if $E$ is a rational function field in one variable over an arbitrary base field, then the localization sequence for $K_2(E)$ (which is completely analogous to (\ref{es})) always splits, thanks to a result of Milnor and Tate (see \cite{Mi1}, Theorem 2.3);
	\item if $E$ is a local field (i.e. a field complete with respect to a discrete valuation whose residue field is finite), then the localization sequence for $K_{2i}(E)$ (which again is completely analogous to (\ref{es})) always splits, thanks to a result of Soulé (see \cite{So}, Proposition 4).
\end{itemize}
Coming back to the case of a number field $F$, consider now the exact sequence which (\ref{es}) induces on $p$-primary parts, which we shall refer to as the localization sequence for $K_{2i}(F)_p$. The problem of the splitting of the localization sequence for $K_{2i}(F)_p$ was first studied by Banaszak: in one of his papers (see \cite{Ba}, Corollary 1) he claims that, if $p$ is odd, the localization sequence for $K_{2i}(F)_p$ splits if and only if $\mathrm{div}(K_{2i}(F))_p=0$ (in fact Banaszak's result is stated in terms of $W\!K_{2i}^{\acute{e}t}(F)$, see Proposition \ref{TBKKSH}). This is obviously a necessary condition, since both the right and the left terms of the localization sequence for $K_{2i}(F)_p$ have trivial subgroup of divisible elements. However the proof of sufficiency seems to be incomplete. It turns out that, for any $i\geq1$, there is a counterexample, namely there is a number field $F$ and a prime $p$ such that $\mathrm{div}(K_{2i}(F))_p=0$ but the localization sequence for $K_{2i}(F)_p$ does not split (see Example \ref{ce}). The counterexample is constructed using Theorem \ref{gh}, which is our main result and is described in Section \ref{mr}: under the assumption that $p$ is odd or $F$ is nonexceptional (see Definition \ref{nonexc}), the obstruction to the existence of a splitting for the localization sequence for $K_{2i}(F)_p$ are the $p$-parts of the $p$-class groups $Cl_{F_{i,n}}^{S_{i,n}}$ of certain subfields $F_{i,n}$ of the fields $F(\mu_{p^n})$, where $\mu_{p^n}$ denotes the group of roots of unity in an algebraic closure of $F$. The main idea of the proof of Theorem \ref{gh} is essentially to use a generalization of an exact sequence of Keune which builds on an earlier result of Tate (Proposition \ref{ke}). For instance, if $\mu_{p^n}\subseteq F$, a consequence of this exact sequence (which has already been used in \cite{Hu}) is an isomorphism
$$(Cl_F^S)_p\otimes \mu_{p^n}\cong \left(K_2(\mathcal{O}_F)\cap p^nK_2(F)\right)\Big/p^nK_2(\mathcal{O}_F)$$
and the triviality of the left-hand term is closely related to the splitting of the localization sequence for $K_{2}(F)_p$. Using a codescent argument, the hypothesis $\mu_{p^n}\subseteq F$ can be removed and this can be generalized to $K_{2i}(F)$ but one has to face the absence of a description of $K_{2i}(F)$ analogous to that of Matsumoto for $K_2(F)$. The way we found to overcome this difficulty is to work with continuous Galois cohomology instead of $K$-theory (the connection between the two theories is classical). This translation implies no further assumptions as explained in Section \ref{cgc}, which is essentially taken from \cite{Ba}, with some slight change and generalization.\\ 
The difference between our splitting criterion and the condition $\mathrm{div}(K_{2i}(F))_p=0$ (which is equivalent to the vanishing of the $i$-th étale wild kernel $W\!K_{2i}^{\acute{e}t}(F)$) is also analyzed at the end of Section \ref{mr}. Anyway the condition $\mathrm{div}(K_{2i}(F))_p=0$ is often necessary and sufficient for the localization sequence for $K_{2i}(F)_p$ to split (for example in the case where $F=\mathbb{Q}$, see Example \ref{qe}).\\

\noindent
\textbf{Acknowledgements}\\ 
I wish to thank Kevin Hutchinson, Jean-François Jaulent and Manfred Kolster for their suggestions and comments.

\section{Localization sequence for continuous Galois cohomology}\label{cgc}

In this section we are going to translate the problem of the existence of a splitting for the localization sequence for $K_{2i}(F)_p$ in cohomological terms. First of all we recall the following notion.

\begin{definition}\label{nonexc} Let $E$ be a field of characteristic other than $2$. Then $E$ is said to be \emph{nonexceptional} if $\mathrm{Gal}(E(\mu_{2^\infty})/E)$ has no element of order $2$ (and \emph{exceptional} otherwise). Here $\mu_{2^{\infty}}$ is the group of roots of unity whose order is a power of $2$ in an algebraic closure of $E$. 
\end{definition}

\begin{remark}
Note that nonexceptional fields have no embeddings in $\mathbb{R}$ (since $\mathbb{R}$ is exceptional and subfields of exceptional fields are exceptional).
\end{remark}

As in the preceding section, $p$ denotes a rational prime, $i$ is a positive integer and $F$ is a number field. We are interested in the localization sequence for $K_{2i}(F)_p$, namely  
\begin{equation}\label{esp}
0\longrightarrow K_{2i}(\mathcal{O}_F)_p\longrightarrow K_{2i}(F)_p \stackrel{\partial_{F,i}}{\longrightarrow} \bigoplus_{v\nmid p\infty} K_{2i-1}(k_v)_p\longrightarrow 0
\end{equation}   
which easily follows from (\ref{es}). This exact sequence has (at least conjecturally, see Remark \ref{ql}) a cohomological counterpart, as shown in the next proposition (which can essentially be found \cite{Ba}, Section I, §2-3). In this paper, cohomology is always continuous cohomology of profinite groups, in the sense of Tate (\cite{Ta2}). 

\begin{notation}\label{sv} 
If $T$ is a finite set of places of $F$, $G_{F,\,T}$ denotes the Galois group of the maximal extension $F_T$ of $F$ unramified ouside $T$ and $\mathcal{O}_F^T$ is the ring of $T$-integers of $F$. For a field $E$, $G_E$ denotes the Galois group of a separable algebraic closure of $E$ (and we use the convention $H^j(E,\,-):=H^j(G_E,\,-)$ for cohomology). Finally, $S_F^{(p)}=S$ denotes the set of primes above $p$ and $\infty$ in $F$.
\end{notation}

For a noetherian $\mathbb{Z}\textstyle{[\frac{1}{p}]}$-algebra $A$ and a natural number $j$, let $K_{j}^{\acute{e}t}(A)$ denote the $j$-th étale $K$-theory group of Dwyer and Friedlander (see \cite{DF}). For any finite set $T$ containing $S$, there are functorial maps 
$$K_{2i}(\mathcal{O}_F^T)_p\stackrel{\nu}{\longrightarrow}K_{2i}^{\acute{e}t}(\mathcal{O}_F^T)\quad \textrm{and}\quad K_{2i}^{\acute{e}t}(\mathcal{O}_F^T)\stackrel{\alpha}{\longrightarrow} H^{2}(G_{F,\,T},\,\mathbb{Z}_p(i+1))$$ 
We know that $\alpha$ is an isomorphism (see \cite{DF}, Remark 8.8). We set $ch=\alpha\circ\nu$ ($ch$ stands for Chern character). 

\begin{prop}[Banaszak]\label{ildiagramma}
Suppose that $p$ is odd or $F$ is nonexceptional. There is a commutative diagram with exact rows as
follows
$$
\xymatrix{
0\ar[r]&K_{2i}(\mathcal{O}_{F}^S)_p  \ar[d]^{ch} \ar[r] & K_{2i}(F)_p \ar[d] \ar[r]^{\partial_{F,i}} &\oplus_{v\nmid p\infty}\,K_{2i-1}(k_v)_p \ar[d]\ar[r]&0\\
0\ar[r]&H^2(G_{F,\,S},\,\mathbb{Z}_p(i+1))\ar[r] & H^{2}(F,\,\mathbb{Z}_p(i+1))_p  \ar[r]^{\partial_{F,i}^{c}} & \oplus_{v\nmid p\infty}\,H^{1}(k_v,\,\mathbb{Z}_p(i))\ar[r]&0\\}
$$
Moreover vertical maps in the above diagram are surjective and the rightmost is an isomorphism.
\end{prop}
\begin{proof}
Let $T$ be any finite set of primes $T$ containing $S$. For the rest of this proof, for any $j\in \mathbb{N}$, set 
$$H^j_T(-):=H^j(G_{F,\,T},\,-)\quad \textrm{and}\quad H^j(-):=H^j(F,\,-)$$
Then, for any $n\in \mathbb{N}$, there is a commutative diagram with exact rows 

$$
\xymatrix{
0\ar[r]&K_{2i}(\mathcal{O}_{F}^S)_p  \ar[d]^{ch} \ar[r] & K_{2i}(\mathcal{O}_{F}^T)_p \ar[d]^{ch} \ar[r]^{\partial_{F,i}} &\oplus_{v\in T\smallsetminus S}\,K_{2i-1}(k_v)_p \ar[d]\ar[r]&0\\
0\ar[r]&H^2_S(\mathbb{Z}_p(i+1))\ar[r] & H^{2}_T(\mathbb{Z}_p(i+1))  \ar[r]^{\partial_{F,i}^{c}} & \oplus_{v\in T\smallsetminus S}\,H^{1}(k_v,\,\mathbb{Z}_p(i))\ar[r]&0\\
0\ar[r]&H^1_{S}(\mathbb{Q}_p/\mathbb{Z}_p(i+1))/\mathrm{Div}\ar[u]^{\delta}\ar[r] & H^{1}_{T}(\mathbb{Q}_p/\mathbb{Z}_p(i+1))/\mathrm{Div}  \ar[u]^{\delta}\ar[r] & \oplus_{v\in T\smallsetminus S}\,H^{0}(k_v,\,\mathbb{Q}_p/\mathbb{Z}_p(i))\ar[u]^{\delta}\ar[r]&0\\}
$$
The definition and the exactness of the middle row can be found in \cite{So}, Section III and and \cite{RW}, Section 4. The upper rightmost vertical map is defined by the diagram (note that in general this does \emph{not} coincide with the direct sum of the residue fields Chern characters) and it is bijective. This follows from the surjectivity of the $ch$'s (see \cite{DF}, Theorem 8.7, for the case $p$ odd and \cite{RW}, Theorem 0.1, for the case $p=2$ and $F$ nonexceptional) together with the easy fact that, for any $v\in T\smallsetminus S$, $H^{1}(k_v,\,\mathbb{Z}_p(i))$ is cyclic of order $|k_v|^i-1$ and has therefore the same order as $K_{2i-1}(k_v)_p$ (thanks to Quillen's calculation). As for the maps denoted with $\delta$, they are connecting homomorphisms in the long exact cohomology sequence relative to the exact sequence 
$$0\to\mathbb{Z}_p(j)\to \mathbb{Q}_p(j)\to \mathbb{Q}_p/\mathbb{Z}_p(j)\to 0$$ 
($j=i,\,i+1$) and they are bijective (see \cite{Ta2}, Proposition 2.3).\\
Taking direct limits as $T\supset S$ grows, we get
\begin{equation}\label{qsz}
\xymatrix{
0\ar[r]&K_{2i}(\mathcal{O}_{F}^S)_p  \ar[d]^{ch\delta^{-1}} \ar[r] & K_{2i}(F)_p \ar[d] \ar[r]^{\partial_{F,i}} &\oplus_{v\nmid p\infty}\,K_{2i-1}(k_v)_p \ar[d]\ar[r]&0\\
0\ar[r]&H^1_S(\mathbb{Q}_p/\mathbb{Z}_p(i+1))/\mathrm{Div}\ar[r] & H^{1}(\mathbb{Q}_p/\mathbb{Z}_p(i+1))/\mathrm{Div}  \ar[r]^{\partial_{F,i}^{c}} & \oplus_{v\nmid p\infty}\,H^{0}(k_v,\,\mathbb{Q}_p/\mathbb{Z}_p(i))\ar[r]&0\\}
\end{equation}
This is because, for any finite set $T$ containing $S$, 
$$\mathrm{Div}(H^{1}_T(\mathbb{Q}_p/\mathbb{Z}_p(i+1)))=\mathrm{Div}(H^{1}(\mathbb{Q}_p/\mathbb{Z}_p(i+1)))$$
(see \cite{Sc}, §4, Lemma 5). Composing again with $\delta$ we get a commutative diagram with exact rows
$$
\xymatrix{
0\ar[r]&K_{2i}(\mathcal{O}_{F}^S)_p  \ar[d]^{ch} \ar[r] & K_{2i}(F)_p \ar[d] \ar[r]^{\partial_{F,i}} &\oplus_{v\nmid p\infty}\,K_{2i-1}(k_v)_p \ar[d]\ar[r]&0\\
0\ar[r]&H^2_S(\mathbb{Z}_p(i+1))\ar[r] & H^{2}(\mathbb{Z}_p(i+1))_p  \ar[r]^{\partial_{F,i}^{c}} & \oplus_{v\nmid p\infty}\,H^{1}(k_v,\,\mathbb{Z}_p(i))\ar[r]&0\\}
$$
Here again the rightmost map is an isomorphism and the other two vertical maps are surjective. 
\end{proof}

\begin{remark}\label{ql} 
The Quillen-Lichtenbaum conjecture predicts that for any finite set of places $T$ containing $S$, 
$$ch:K_{2i}(\mathcal{O}_{F}^T)_p\longrightarrow H^2(G_{F,\,T},\,\mathbb{Z}_p(i+1))$$ 
is indeed an isomorphism. Tate (see \cite{Ta2}, Theorem 5.4) proved this statement for $i=1$: in this case $ch$ is just the Galois symbol (\cite{DF}, proof of Theorem 8.2). The general case of the Quillen-Lichtenbaum conjecture is a consequence of the Bloch-Kato conjecture on the Galois symbol whose proof should appear soon.
\end{remark}

In the following we will be mainly interested in the bottow row of the diagram of Proposition \ref{ildiagramma}, namely the exact sequence 
\begin{equation}\label{ces}
0\to H^2(G_{F,\,S},\,\mathbb{Z}_p(i+1))\to H^{2}(F,\,\mathbb{Z}_p(i+1))_p\to\displaystyle{\bigoplus_{v\nmid p\infty}}\,H^{1}(k_v,\,\mathbb{Z}_p(i))\to0
\end{equation}
We will refer to it as the localization sequence for $H^{2}(F,\,\mathbb{Z}_p(i+1))$. The next proposition shows that, even without using the Quillen-Lichtenbaum conjecture, for our purposes there is no difference between considering (\ref{esp}) or (\ref{ces}). 

\begin{definition}
A subgroup $B$ of an abelian group $A$ is \emph{pure} if $nA\cap B=nB$ for each $n\in\mathbb{N}$.
\end{definition}

\begin{prop}
Suppose that $p$ is odd or $F$ is nonexceptional. Then the localization sequence for $H^{2}(F,\,\mathbb{Z}_p(i+1))$ splits if and only the localization sequence for $K_{2i}(F)_p$ splits. 
\end{prop}
\begin{proof}
A diagram chasing in the diagram of Proposition \ref{ildiagramma} shows that if the localization sequence for $K_{2i}(F)_p$ splits then the localization sequence for $H^{2}(F,\,\mathbb{Z}_p(i+1))$ splits. As for the converse, we can assume $p\ne2$ because if $p=2$, then $F$ is totally imaginary and the $ch$'s are isomorphisms by \cite{RW}, Theorem 0.1\footnote{It is likely that there exists a proof of the proposition for $p=2$ and $F$ nonexceptional which does not use the results in \cite{RW}.}. Banaszak (see \cite{Ba}, Proposition 2) proved that the natural map
\begin{equation}\label{ssban}
\nu:K_{2i}(\mathcal{O}_F)_p\longrightarrow K_{2i}^{\acute{e}t}(\mathcal{O}_F\textstyle{[\frac{1}{p}]})
\end{equation} 
is split surjective. We are going to prove the analogous result for the map 
\begin{equation}\label{sscap}
\nu: K_{2i}(F)_p\longrightarrow K_{2i}^{\acute{e}t}(F)_p
\end{equation}
with the same strategy as Banaszak, taking into account that the groups involved are no longer finite (but still torsion). 
First of all, for each $n\in \mathbb{N}$, there is a commutative diagram 
$$
\begin{CD}
K_{2i+1}(F,\,\mathbb{Z}/p^n\mathbb{Z})@>>> K_{2i}(F)[p^n]@>>>0\\
 @VVV   @V\nu VV \\
K_{2i+1}^{\acute{e}t}(F,\,\mathbb{Z}/p^n\mathbb{Z})@>>> K_{2i}^{\acute{e}t}(F)[p^n]@>>>0\\
\end{CD}
$$
with exact rows and surjective vertical maps (see \cite{Ba}, Diagram 1.6). This implies that the kernel $C_i$ of the map $\nu$ in (\ref{sscap}) is a pure subgroup of $K_{2i}(F)_p$. Moreover $C_i$ is finite since it is easily seen to coincide with the kernel of the map in (\ref{ssban}) (see \cite{Ba}, Remark 7). Hence Theorem 7 of \cite{Ka} tells us that the map in (\ref{sscap}) is split. By definition of the maps $ch$, this proves that the map 
$$K_{2i}(F)_p\longrightarrow H^2(F,\,\mathbb{Z}_p(i+1))_p$$
in the diagram od Proposition \ref{ildiagramma} splits. The proof of the proposition is then achieved by a diagram chasing in the diagram of Proposition \ref{ildiagramma}. 
\end{proof}

\section{Main result}\label{mr}
We are going to describe the obstruction to the existence of a splitting for the localization sequence for $H^2(F,\mathbb{Z}_p(i+1))$ in terms of coinvariants of twisted $p$-parts of the class groups of certain subfields of the fields $F(\mu_{p^n})$. For a field $E$, we denote by $\mu_{p^n}(E)$ the group of $p^n$-th roots of unity in an algebraic closure of $E$ (the reference to $E$ in $\mu_{p^n}(E)$ will be often omitted).\\

\begin{notation}
For typographical convenience, we set 
$$\Omega_{F,i,n}^{(p)}=\Omega_{i,n}^{(p)}=\left(H^2(G_{F,\,S},\,\mathbb{Z}_p(i+1))\cap p^nH^{2}(F,\,\mathbb{Z}_p(i+1))_p\right)\Big/p^nH^2(G_{F,\,S},\,\mathbb{Z}_p(i+1))$$
\end{notation}

\begin{remark}
Note that, from the definition of $\Omega^{(p)}_{i,n}$, we have $\Omega^{^{(p)}}_{i,0}=0$. 
\end{remark}

The following lemma shows that the $\Omega_{i,n}^{(p)}$'s are the obstructions to the existence of a splitting for the localization sequence for $H^2(F,\mathbb{Z}_p(i+1))$. 

\begin{lemma}\label{snake}
The localization sequence for $H^{2}(F,\,\mathbb{Z}_p(i+1))$ splits if and only if for every $n\in \mathbb{N}$ we have $\Omega^{(p)}_{i,n}=0$.
\end{lemma}
\begin{proof}
Thanks to Theorem 5 in \cite{Ka} and the fact that any direct summand of an abelian group is pure, we have the following equivalences 
\begin{eqnarray*}
\Omega^{(p)}_{i,n}=0 \quad\forall\,n\in \mathbb{N}&\Leftrightarrow& \textrm{$H^2(G_{F,\,S},\,\mathbb{Z}_p(i+1))$ is pure in $H^{2}(F,\,\mathbb{Z}_p(i+1))_p$}\\
&\Leftrightarrow& \textrm{the localization sequence for $H^{2}(F,\,\mathbb{Z}_p(i+1))$ splits.}
\end{eqnarray*}

\end{proof}

We will make use of the following notation (see \cite{We2}).

\begin{notation}
Let $E$ be any field. If $M$ is a $G_E$-module, we denote by $E(M)$ the fixed field of the kernel of the homomorphism $G_E\to \mathrm{Aut}(M)$ induced by the action of $\mathrm{Gal}(\overline{E}/E)$ on $M$. 
\end{notation}

For each $n\in\mathbb{N}$, we now introduce the subfield $F_{i,n}$ of $F(\mu_{p^n})$ which will be relevant for us. Such subfields has been used for the first time by Weibel in \cite{We2}.
\begin{notation}
Let $F$ be a number field. For $n\in \mathbb{N}$ and $i\geq 1$, set $F_{i,n}=F(\mu_{p^n}(F)^{\otimes i})$ (note that $F_{i,0}=F$ for any $i\geq 1$). We will also use the notation $\Gamma_{i,n}=\mathrm{Gal}(F_{i,n}/F)$. If $w$ is a place in $F_{i,n}$, then denote by $(k_{i,n})_w$ the residue field of $F_{i,n}$ at $w$ (thus $(k_{i,0})_w=k_w$). Finally, let $S_{i,n}$ be the set of primes of $F_{i,n}$ above $p$ and $\infty$ (thus $S_{i,0}=S$) and let $G_{S_{i,n}}$ denote the Galois group of the maximal extension of $F_{i,n}$ unramified outside $S_{i,n}$ (thus $G_{S_{i,0}}=G_S=G_{F,\,S}$). Of course, in all this notation, a reference to $p$ should appear but the context should prevent any misunderstanding.  
\end{notation}

The following lemmas are technical ingredients of the proof of Theorem \ref{gh}. Most of them are well-known, but we collect them for the convenience of the reader, giving proofs whenever we could not find an adequate reference in the literature.
  
\begin{lemma}\label{pn}
Suppose that $p$ is an odd prime or $F$ is nonexceptional. Let $v\nmid p$ be a place in $F$. For every $n,\,m\in \mathbb{N}$, there are isomorphisms of $\Gamma_{i,m}$-modules  
$$
\bigoplus_{w|v} H^{1}((k_{i,m})_w,\,\mathbb{Z}_p(i))[p^n]\cong \bigoplus_{w|v}H^0((k_{i,m})_w,\,\mu_{p^{n}}^{\otimes i})
$$
and 
$$
H^{2}(G_{S_{i,m}},\,\mathbb{Z}_p(i))/p^n\cong H^{2}(G_{S_{i,m}},\,\mu_{p^{n}}^{\otimes i})
$$
\end{lemma}
\begin{proof}
Both assertions come from the cohomology sequence corresponding to the exact sequence 
$$0\to \mathbb{Z}_p(i)\stackrel{p^n}{\longrightarrow} \mathbb{Z}_p(i)\longrightarrow \mu_{p^n}^{\otimes i}\to 0$$
In fact 
\begin{enumerate}
	\item[(i)] $H^{0}((k_{i,m})_w,\,\mathbb{Z}_p(i))=0$ for any place $w|v$ in $F_m$ (since $i>0$);
	\item[(ii)] $H^{3}(G_{S_{i,m}},\,\mathbb{Z}_p(i))=0$ since $G_{S_{i,m}}$ has $p$-cohomological dimension less or equal to $2$ (see \cite{NSW}, (8.3.17) and (8.3.18)): this is true even in the case $p=2$ because in that case $F_{i,m}$ has to be nonexceptional and therefore it has no real embeddings.
\end{enumerate}
Now, since conjugation in cohomology commutes with the connecting homomorphism, we have that, for any fixed $w_0|v$ in $F_{i,m}$, (i) gives an isomorphism of $D_v$-modules ($D_v$ being the decomposition group at $v$ in $F_{i,m}/F$) 
$$H^{1}((k_{i,m})_{w_0},\,\mathbb{Z}_p(i))[p^n]\cong H^0((k_{i,m})_{w_0},\,\mu_{p^{n}}^{\otimes i})$$
and (ii) gives an isomorphism of $\Gamma_{i,m}$-modules 
$$H^{2}(G_{S_{i,m}},\,\mathbb{Z}_p(i))/p^n\cong H^{2}(G_{S_{i,m}},\,\mu_{p^{n}}^{\otimes i})$$  
Furthermore 
$$\bigoplus_{w|v}H^{1}((k_{i,m})_w,\,\mathbb{Z}_p(i))[p^n]=\textrm{Ind}_{D_v}^{\Gamma_m}H^{1}((k_{i,m})_{w_0},\,\mathbb{Z}_p(i))[p^n]$$ 
for any fixed $w_0|v$ in $F_{i,m}$ and 
$$\bigoplus_{w|v}H^0((k_{i,m})_w,\,\mu_{p^{n}}^{\otimes i})=\textrm{Ind}_{D_v}^{\Gamma_n}H^0((k_{i,m})_{w_0},\,\mu_{p^{n}}^{\otimes i})$$ 
This concludes the proof.
\end{proof}

\begin{lemma}\label{cf}
Let $v\nmid p$ be a finite place of $F$. If $w$ is a place of $F_{i,n}$ above $v$, then $(k_{i,n})_w=k_v(\mu_{p^n}^{\otimes i})$.
\end{lemma}
\begin{proof}
Both $(k_{i,n})_w$ and $k_v(\mu_{p^n}^{\otimes i})$ are subextensions of $k_v(\mu_{p^n})/k_v$. Set $D_v$ (resp. $D'_v$) for the decomposition group of $v$ in $F_{i,n}/F$ (resp. $F(\mu_{p^n})/F$) and $P$ for the Galois group of $F(\mu_{p^n})/F_{i,n}$. Identifying $\mathrm{Gal}(k_v(\mu_{p^n})/k_v)$ with $D'_v$ we get
$$\mathrm{Gal}(k_v(\mu_{p^n})/k_v(\mu_{p^n}^{\otimes i}))= D'_v\cap P$$
$$\mathrm{Gal}((k_{i,m})_v/k_v)= D_v=D'_v/D'_v\cap P$$
which proves the result.
\end{proof}

\begin{lemma}\label{we}
Let $G$ be a finite cyclic group and let $M$ be a faithful $G$-module which is cyclic of order a power of $p$ as an abelian group. Then
\begin{itemize}
	\item if $p\ne 2$, then $H_j(G,\,M)=0$ for any $j\geq 1$ (or equivalently the norm map $M_G\rightarrow M^G$ is surjective);
	\item if $p=2$ and $G\ne\{\pm1\}$ (i.e. $G$ is not a group of order $2$ whose generator $g$ satisfies $gm=-m$ for any $m\in M$), then $H_j(G,\,M)=0$ for any $j\geq 1$ (or equivalently the norm map $M_G\rightarrow M^G$ is surjective);
\end{itemize}
\end{lemma}
\begin{proof}
See \cite{We2}, Lemma 3.2 and Remark 3.2.1.
\end{proof}

\begin{lemma}\label{surj}
Let $v\nmid p$ be a finite place of $F$. If $w$ is a place of $F_{i,n}$ above $v$, then the corestriction homomorphisms 
$$H^0((k_{i,n})_w,\,\mu_{p^{n}}^{\otimes i})\longrightarrow H^0(k_v,\,\mu_{p^{n}}^{\otimes i})$$ 
are surjective if $p$ is odd or $n$ is large enough.
\end{lemma}
\begin{proof}
We have 
$$H^0((k_{i,n})_w,\,\mu_{p^{n}}^{\otimes i})=\mu_{p^{n}}^{\otimes i}\quad\textrm{and}\quad H^0(k_v,\,\mu_{p^{n}}^{\otimes i})=H^0((k_{i,n})_w,\,\mu_{p^{n}}^{\otimes i})^{\mathrm{Gal}((k_{i,n})_w/k_v)}=(\mu_{p^{n}}^{\otimes i})^{\mathrm{Gal}((k_{i,n})_w/k_v)}$$ 
for any $n\in \mathbb{N}$. In particular the corestriction coincides with the norm $\mu_{p^{n}}^{\otimes i}\to(\mu_{p^{n}}^{\otimes i})^{\mathrm{Gal}((k_{i,n})_w/k_v)}$. We then conclude by Lemma \ref{we} with $M=\mu_{p^{n}}^{\otimes i}$ and $G=\mathrm{Gal}((k_{i,n})_w/k_v)$ whose hypotheses are satisfied since $\mathrm{Gal}((k_{i,n})_w/k_v)$ is always cyclic (since finite fields of odd characteristic are nonexceptional), acts faithfully on $\mu_{p^{n}}^{\otimes i}$ by Lemma \ref{cf} and is different from $\{\pm1\}$ if $n$ is large enough.   
\end{proof}

\begin{lemma}\label{we1}
Let $c:\bigoplus_{w|p}\mu_{p^n}^{\otimes i}(F)\to \mu_{p^{n}}^{\otimes i}(F)$ be the codiagonal map $(\zeta_w)_{w} \mapsto \prod_{w}\zeta_w$. Then $c$ is a surjective map of $\Gamma_{i,n}$-modules and, for any $j\geq 1$, $H_j(\Gamma_{i,n},\,\mathrm{Ker}\, c)=0$ if $p$ is odd or $n$ is large enough.  
\end{lemma}
\begin{proof}
If $D_v$ denotes the decomposition group at $v$ in $F_{i,n}/F$, we have 
$$H_j(\Gamma_{i,n},\,\bigoplus_{w|p}\mu_{p^n}^{\otimes i}) =H_j(\Gamma_{i,n},\,\bigoplus_{v|p}\mathrm{Ind}_{D_v}^{\Gamma_{i,n}}\mu_{p^n}^{\otimes i}) =\bigoplus_{v|p}H_j(D_v,\,\mu_{p^n}^{\otimes i})$$
by Shapiro's lemma, for any $j\geq0$. Moreover if $p$ is odd or $n$ is large enough 
$$H_j(\Gamma_{i,n},\,\mu_{p^n}^{\otimes i})=0=H_j(D_v,\,\mu_{p^n}^{\otimes i})$$
for any $j\geq 1$ by Lemma \ref{we} and the result follows by considering the $\Gamma_{i,n}$-homology sequence of the exact sequence $$0\to \mathrm{Ker}\,c \to\bigoplus_{w|p}\mu_{p^n}^{\otimes i}(F)\stackrel{c}{\longrightarrow} \mu_{p^{n}}^{\otimes i}(F)\to 0$$
\end{proof}

\begin{lemma}\label{we2}
Suppose that $p$ is an odd prime or $F$ is nonexceptional. For every $n\in\mathbb{N}$, the corestriction map induces an isomorphism $$H^{2}(G_{S_{i,n}},\,\mu_{p^{n}}^{\otimes i})_{\Gamma_{i,n}}\stackrel{\sim}{\longrightarrow}H^{2}(G_{S},\,\mu_{p^{n}}^{\otimes i})$$
\end{lemma}
\begin{proof}
Apply \cite{We2}, Proposition 2.2, using the fact that $G_S$ has $p$-cohomological dimension less or equal to $2$ (see the proof of Lemma \ref{pn}).
\end{proof}

\begin{notation}
For any $i\geq1$ and any $n\in \mathbb{N}$, we will denote by $Cl^{S_{i,n}}_{F_{i,n}}$ the class group of $\mathcal{O}_{F_{i,n}}^{S_{i,n}}$ (see Notation \ref{sv}). We will also use the notation $Cl_F^S$ for $Cl^{S_{i,0}}_{F_{i,0}}$. 
\end{notation}

\begin{remark}
Note that if $p$ is odd or $F$ is nonexceptional, $Cl^{S_{i,n}}_{F_{i,n}}$ is also equal to the quotient of the narrow class group of $F_{i,n}$ by the subgroup generated by the classes of finite primes in $S_{i,n}$. 
\end{remark}

The following result is the main ingredient of the proof of Theorem \ref{gh}. 

\begin{prop}[Tate-Keune-Weibel]\label{ke}
For every $n\in \mathbb{N}$ and every $i\geq 1$ there is an exact sequence of $\Gamma_{i,n}$-modules as follows 
$$0\to Cl^{S_{i,n}}_{F_{i,n}}\otimes\mu_{p^n}^{\otimes i}\stackrel{J_{i,n}}{\longrightarrow}
H^{2}(G_{S_{i,n}},\,\mu_{p^{n}}^{\otimes i+1})\stackrel{inv}{\longrightarrow} \displaystyle{\bigoplus_{w|p}\mu_{p^n}^{\otimes i}}\stackrel{c}{\longrightarrow} 
\mu_{p^{n}}^{\otimes i}\to 
0$$
where $c$ is the map of Lemma \ref{we1} and $inv$ is the twist by $\mu_{p^n}^{\otimes i}$ of the direct sum of the local invariant maps (for the definition of $J_{i,n}$, see Remark \ref{kum} below). Moreover, if $p$ is odd or $n$ is large enough, taking coinvariants by $\Gamma_{i,n}$ gives
$$0\to \left(Cl^{S_{i,n}}_{F_{i,n}}\otimes\mu_{p^n}^{\otimes i}\right)_{\Gamma_{i,n}}\stackrel{J_{i,n}}{\longrightarrow}
H^{2}(G_{S},\,\mu_{p^{n}}^{\otimes i+1})\stackrel{inv}{\longrightarrow} \displaystyle{(\bigoplus_{w|p}\mu_{p^n}^{\otimes i})_{\Gamma_{i,n}}\stackrel{c}{\longrightarrow} 
(\mu_{p^{n}}^{\otimes i})_{\Gamma_{i,n}}\to 
0}$$
\end{prop}
\begin{proof}
The first sequence is exact by Proposition 4.1 of \cite{We2}. The second follows from the first by taking $\Gamma_{i,n}$-coinvariants and using Lemma \ref{we1}, Lemma \ref{we2} and $H_1(\Gamma_{i,n},\,\mu_{p^{n}}^{\otimes i})=0$ (by Lemma \ref{we}).
\end{proof}

\begin{remark}\label{kum} 
We briefly recall the definition of the map $J_{i,n}$ since it will be used later on. To ease notation, set $E_{i,n}=(F_{i,n})_{S_{i,n}}$ (see Notation \ref{sv}) and denote by $\mathcal{O}_{E_{i,n}}$ the ring of integers of $E_{i,n}$. The first ingredient we need is the isomorphism (see \cite{NSW}, Proposition 8.3.10)
$$f:Cl_{F_{i,n}}^{S_{i,n}}/p^n\stackrel{\sim}{\longrightarrow} H^1(G_{S_{i,n}},\,\mathcal{O}_{E}^\times)/p^n$$
which is defined as follows: if $\mathfrak{a}$ is a fractional ideal of $F_{i,n}$ (not divisible by $p$), then $f([\mathfrak{a}])$ is represented by the continuous $1$-cocycle $\sigma\mapsto \sigma(\alpha)/\alpha$ where $\alpha\in E$ is such that $\alpha\mathcal{O}_E=\mathfrak{a}\mathcal{O}_E$. 
Now there is an exact sequence (see \cite{NSW}, Proposition 8.3.3)
$$0\to \mu_{p^n}\to \mathcal{O}_{E}^\times \stackrel{p^n}{\longrightarrow} \mathcal{O}_{E}^\times \to 0$$
Taking $G_{S_{i,n}}$-cohomology gives an injective map 
$$g:H^1(G_{S_{i,n}},\,\mathcal{O}_{E}^\times)/p^n\to H^2(G_{S{i,n}},\,\mu_{p^n})$$
Set $g\circ f=J'_{i,n}$: then $J_{i,n}$ is just $J'_{i,n}$ twisted by $\mu_{p^n}^{\otimes i}$.  
\end{remark}

The following two lemmas clarify the connection between $\Omega_{i,n}^{(p)}$ whose precise description is given in Theorem \ref{gh}.  

\begin{lemma}\label{dic}
Suppose that $p$ is an odd prime or $F$ is nonexceptional. Then, for each $n\in \mathbb{N}$ and each $i,\,j\geq 1$, the following diagram is commutative
\begin{equation}\label{gigino}
\begin{CD}
\displaystyle{\bigoplus_{w\nmid p\infty}H^{1}((k_{i,n})_w,\,\mathbb{Z}_p(j))[p^n]} @>I^c_{F_{i,n},j,n}>> H^2(G_{S_{i,n}},\,\mathbb{Z}_p(j+1))/p^n\\
@A AA @V VV\\
\displaystyle{\bigoplus_{w\nmid p\infty}H^{0}((k_{i,n})_w,\,\mu_{p^n}^{\otimes j})} @>-T_{F_{i,n},j,n}>> H^2(G_{S_{i,n}},\,\mu_{p^n}^{\otimes j+1})\\
\end{CD}
\end{equation}
Here the left-hand and the right vertical maps are the isomorphisms of Lemma \ref{pn}, the map $I^c_{i,j,n}$ is induced by the snake lemma on the exact sequence (\ref{ces}) and the map $T_{F_{i,n},j,n}$ is the one appearing in the long exact localization sequence
\begin{equation}\label{mupnls}
0\to H^1(G_{S_{i,n}},\,\mu_{p^n}^{\otimes j+1})\to H^1(F_{i,n},\,\mu_{p^n}^{\otimes j+1})\to \bigoplus_{w\nmid p\infty}H^{0}((k_{i,n})_w,\,\mu_{p^n}^{\otimes j})\stackrel{T_{F_{i,n},j,n}}{\longrightarrow}H^2(G_{S_{i,n}},\,\mu_{p^n}^{\otimes j+1})\to \ldots
\end{equation} 
(see \cite{So}, Proposition 1).
\end{lemma}
\begin{proof}
It is easy to see, using the results of \cite{Ta2}, § 2, that the commutativity of (\ref{gigino}) is equivalent to the commutativity of the following diagram
\begin{equation}\label{gigetto}
\begin{CD}
\displaystyle{\bigoplus_{w\nmid p\infty}H^{0}((k_{i,n})_w,\,\mathbb{Q}_p/\mathbb{Z}_p(j))[p^n]} @>>> H^1(G_{S_{i,n}},\,\mathbb{Q}_p/\mathbb{Z}_p(j+1))/p^n\\
@A AA @V VV\\
\displaystyle{\bigoplus_{w\nmid p\infty}H^{0}((k_{i,n})_w,\,\mu_{p^n}^{\otimes j})} @>-T_{F_{i,n},j,n}>> H^2(G_{S_{i,n}},\,\mu_{p^n}^{\otimes j+1})\\
\end{CD}
\end{equation}
Here the left-hand vertical map is induced by the injection $\mu_{p^{n}}^{\otimes j}\to\mathbb{Q}_p/\mathbb{Z}_p(j)$ while the right-hand vertical map is the connecting homomorphisms for the exact sequence 
$$0\to\mu_{p^n}^{\otimes j+1}\to\mathbb{Q}_p/\mathbb{Z}_p(j+1)\stackrel{p^n}{\longrightarrow} \mathbb{Q}_p/\mathbb{Z}_p(j+1)\to 0$$
The upper horizontal map is induced by the snake lemma on the exact sequence at the bottom of diagram (\ref{qsz}).\\
Now recall that the exact sequence (\ref{mupnls}) comes from the Leray spectral sequence
$$H^{s}_{\acute{e}t}(\mathcal{O}_{F_{i,n}}^{S_{i,n}},\,R^t\iota_{\ast}\mu_{p^n}^{\otimes j+1})\Rightarrow H^{s+t}_{\acute{e}t}(F_{i,n},\,\mu_{p^n}^{\otimes j+1})$$
where $\iota: \mathrm{Spec}(F_{i,n})\to \mathrm{Spec}(\mathcal{O}_{F_{i,n}}^{S_{i,n}})$ is the morphism of schemes induced by $\mathcal{O}_{F_{i,n}}^{S_{i,n}}\hookrightarrow F_{i,n}$. The above spectral sequence 
can be rewritten as an Hochschild-Serre spectral sequence in Galois cohomology
$$H^{s}(G_{S_{i,n}},\,H^t(H_{S_{i,n}},\,\mu_{p^n}^{\otimes j+1}))\Rightarrow H^{s+t}(F_{i,n},\,\mu_{p^n}^{\otimes j+1})$$
where $H_{S_{i,n}}$ is the absolute Galois group of $(F_{i,n})_{S_{i,n}}$. Then the map $T_{i,j,n}$ becomes the transgression
$$tg:H^{0}(G_{S_{i,n}},\,H^1(H_{S_{i,n}},\,\mu_{p^n}^{\otimes j+1}))\to H^2(G_{S_{i,n}},\,\mu_{p^n}^{\otimes j+1})$$
Similarly, the upper horizontal map in (\ref{gigetto}) becomes the one induced by the snake lemma on the exact sequence
$$
0\to H^1(G_{S_{i,n}},\,\mathbb{Q}_p/\mathbb{Z}_p(j+1))/\mathrm{Div}\to H^{1}(F_{i,\,n},\,\mathbb{Q}_p/\mathbb{Z}_p(j+1))/\mathrm{Div}\to H^0(G_{S_{i,n}},\,H^1(H_{i,n},\,\mathbb{Q}_p/\mathbb{Z}_p(j+1))\to0
$$
(compare with the bottom row of diagram (\ref{qsz})). In other words the commutativity of (\ref{gigetto}) is equivalent to the commutativity of
$$
\begin{CD}
H^{0}(G_{S_{i,n}},\,H^1(H_{S_{i,n}},\,\mathbb{Q}_p/\mathbb{Z}_p(j+1)))[p^n] @>>> H^1(G_{S_{i,n}},\,\mathbb{Q}_p/\mathbb{Z}_p(j+1))/p^n\\
@A AA @V VV\\
H^{0}(G_{S_{i,n}},\,H^1(H_{S_{i,n}},\,\mu_{p^n}^{\otimes j+1})) @>-tg>> H^2(G_{S_{i,n}},\,\mu_{p^n}^{\otimes j+1})\\
\end{CD}
$$
Using the description of $tg$ given in \cite{NSW}, Proposition (1.6.5), it is not difficult to see that the above diagram is indeed commutative. We give here the explicit computation. Call
$\alpha$ the map obtained by composing the left vertical arrow, the upper horizontal arrow and the right vertical arrow in the above diagram. We have to prove that $\alpha=-tg$. Let $f:H_{S_{i,n}}\to \mu_{p^n}^{\otimes j+1}$ be a $1$-cocycle whose class in $H^1(H_{S_{i,n}},\,\mu_{p^n}^{\otimes j+1}))$ is fixed by $G_{S_{i,n}}$. Composing $f$ with the natural injection $\mu_{p^n}^{\otimes j+1}\to\mathbb{Q}_p/\mathbb{Z}_p(j+1)$ we get a $1$-cocycle of $H_{S_{i,n}}$ with values in $\mathbb{Q}_p/\mathbb{Z}_p(j+1)$ which we also call $f$. Clearly the class of $f$ in $H^1(H_{S_{i,n}},\,\mathbb{Q}_p/\mathbb{Z}_p(j+1))$ is fixed by $G_{S_{i,n}}$ and annihilated by $p^n$. Then we can choose a $1$-cocycle $f':G_{F_{i,n}}\to \mathbb{Q}_p/\mathbb{Z}_p(j+1)$ such that $f'|_{H_{i,n}}=f$ and there exists a $1$-cocycle $f'':G_{S_{i,n}}\to \mathbb{Q}_p/\mathbb{Z}_p(j+1)$ such that $p^nf'=f''\pi$ where $\pi:G_{F_{i,n}}\to G_{S_{i,n}}$ is the natural projection. If $s:G_{S_{i,n}}\to G_{F_{i,n}}$ is a (continuous) section of $\pi$, we have $p^nf's=f''$. Finally, then, $\alpha([f])\in H^2(G_{S_{i,n}},\,\mu_{p^n}^{\otimes j+1})$ is represented by the $2$-cocycle
$$(g_1,\,g_2)\mapsto \partial(f's)(g_1,\,g_2)=g_1f's(g_2)-f's(g_1g_2)+f's(g_1)=-f's(g_1g_2)+f'(s(g_1)s(g_2))$$
(the last equality comes from the fact that $f'$ is a cocycle). On the other hand, since any element of $G_{F_{i,n}}$ can be written uniquely as $s(g)h$ for some $g\in G_{S_{i,n}}$ and $h\in H_{S_{i,n}}$, we can define a map $\widetilde{f}:G_{F_{i,n}}\to\mu_{p^n}^{\otimes j+1}$ by $\widetilde{f}(s(g)h)=s(g)f(h)$. Then $tg[f]\in H^2(G_{S_{i,n}},\,\mu_{p^n}^{\otimes j+1})$ is represented by the $2$-cocycle (see \cite{NSW}, Proposition (1.6.5) and its proof)
\begin{eqnarray*}
(g_1,\,g_2)\mapsto \partial \widetilde{f}(s(g_1),\,s(g_2))&=& s(g_1)\widetilde{f}(s(g_2))-\widetilde{f}(s(g_1)s(g_2))+\widetilde{f}(s(g_1))\\
&=&-\widetilde{f}(s(g_1)s(g_2))=-s(g_1g_2)f'\left(\frac{s(g_1)s(g_2)}{s(g_1g_2)}\right)
\end{eqnarray*}
(the last equality comes from the definition of $\widetilde{f}$ and the fact that $f'|_{H_{i,n}}=f$). The result follows because $f'$ is a cocycle. 
\end{proof} 

\begin{remark}
I wonder if there exists a proof of the preceding lemma which does not make use of explicit calculation. The fact that $-T_{F_{i,n},j,n}$ (and not simply $T_{F_{i,n},j,n}$) is necessary to have a commutative diagram is somehow unexpected. 
\end{remark}

\begin{lemma}\label{maria}
For each $n\in \mathbb{N}$ and each $i\geq1$, the following diagram is commutative
$$
\begin{diagram}
\bigoplus_{w\nmid p\infty}\mu_{p^n}^{\otimes i} & &\rTo^{T_{F_{i,n},i,n}}& &H^2(G_{S_{i,n}},\,\mu_{p^n}^{\otimes i+1})\\
&\rdTo^{\pi_{i,n}} & &\ruTo^{J_{i,n}} \\
& &Cl^{S_{i,n}}_{F_{i,n}}\otimes \mu_{p^n}^{\otimes i}
\end{diagram}
$$
where, if $\mathfrak{P}_w$ is the ideal representing $w$ in $F_{i,n}$, we set 
$$\pi_{i,n}\left((\zeta_w)_w\right)=[\prod_{w\nmid p\infty}\mathfrak{P}_w\otimes \zeta_w]$$
\end{lemma}
\begin{proof}
The commutativity of the diagram in the statement is equivalent to the commutativity of the following diagram
$$
\begin{diagram}
\bigoplus_{w\nmid p\infty}\mathbb{Z}/p^n\mathbb{Z} & &\rTo^{T_{F_{i,n},0,n}}& &H^2(G_{S_{i,n}},\,\mu_{p^n})\\
&\rdTo^{\pi'_{i,n}} & &\ruTo^{J'_{i,n}} \\
& &Cl^{S_{i,n}}_{F_{i,n}}\otimes \mathbb{Z}/p^n\mathbb{Z}
\end{diagram}
$$
where the definition of $\pi'_{i,n}$ is similar to that of $\pi_{i,n}$ and $J'_{i,n}$ is the map of Remark \ref{kum}. 
Comparing the Leray spectral sequence and the Hochschild-Serre spectral sequence and using the computation of \cite{So}, III.1.3, we get a commutative diagram
$$
\begin{CD}
\displaystyle{\bigoplus_{w\nmid p\infty}\mathbb{Z}/p^n\mathbb{Z}}@>T_{F_{i,n},0,n}>> H^2(G_{S_{i,n}},\,\mu_{p^n})\\
@A\Sigma AA @|\\
H^0(G_{S_{i,n}},\,H^1(H_{S_{i,n}},\,\mu_{p^n}))@>tg>>H^2(G_{S_{i,n}},\,\mu_{p^n})
\end{CD}
$$
where $\Sigma$ is the isomorphism defined in the following way. For each $w\nmid p\infty$, let $(F_{i,n})_w$ be a completion of $F_{i,n}$ at $w$ and let $(F_{i,n})_w^{un}$ be the maximal unramified extension of $(F_{i,n})_w$ inside an algebraic closure $\overline{(F_{i,n})_w}$ of $(F_{i,n})_w$. Once an embedding of an algebraic closure of $F_{i,n}$ in $\overline{(F_{i,n})_w}$ has been chosen, we get a restriction map $\rho_w:\mathrm{Gal}(\overline{(F_{i,n})_w}/(F_{i,n})_w^{un})\to H_{S_{i,n}}$. Now, given $f\in H^1(H_{S_{i,n}},\,\mu_{p^n})^{G_{S_{i,n}}}=\mathrm{Hom}_{G_{S_{i,n}}}(H_{S_{i,n}},\,\mu_{p^n})$, we define 
$$f\mapsto(f\circ \rho_w)_w\in \bigoplus_{w\nmid p\infty} \mathrm{Hom}(\mathrm{Gal}(\overline{(F_{i,n})_w}/(F_{i,n})_w^{un}),\,\mu_{p^n})=\bigoplus_{v\nmid p\infty}H^1((F_{i,n})_w^{un},\,\mu_{p^n})$$
The definition of $\Sigma$ now results via Kummer theory, since 
$$H^1((F_{i,n})_w^{un},\,\mu_{p^n})\cong ((F_{i,n})_w^{un})^\times/((F_{i,n})_w^{un})^{\times p^n}=\langle\pi_w\rangle/\langle\pi_w^{p^n}\rangle=\mathbb{Z}/p^n\mathbb{Z}$$
where $\pi_w\in (F_{i,n})_w$ is any uniformizer. We are then left to prove that $J'_{i,n}\pi'_{i,n}\Sigma=tg$.\\
Let $f\in H^1(H_{S_{i,n}},\,\mu_{p^n})^{G_{S_{i,n}}}=\mathrm{Hom}_{G_{S_{i,n}}}(H_{S_{i,n}},\,\mu_{p^n})$ and let $\beta \in F_{S_{i,n}}^\times$ be an element representing $f$ via the Kummer isomorphism $F_{S_{i,n}}^\times/(F_{S_{i,n}}^\times)^{p^n}\cong H^1(H_{S_{i,n}},\,\mu_{p^n})$, i.e.  
$$f(h)=\frac{h(\beta^{\frac{1}{p^n}})}{\beta^{\frac{1}{p^n}}} \quad\textrm{for any $h\in H_{S_{i,n}}$}$$ 
It is easy to see that $\pi'_{i,n}\Sigma(f)=[\prod_{w\nmid p\infty} \mathfrak{P}_w^{\tilde{w}(\beta)}]$
where $\tilde{w}$ is the extension of the valuation $w$ to $F_{S_{i,n}}$ defined by the chosen embedding $F_{S_{i,n}}\hookrightarrow \overline{(F_{i,n})_w}$. Thanks to the description of $J'_{i,n}$ given in Remark \ref{kum}, we see that $J'_{i,n}\pi'_{i,n}\Sigma(f)\in H^2(G_{S_{i,n}},\,\mu_{p^n})$ is the class represented by the inhomogenous $2$-cocycle
$$(\sigma,\,\tau) \mapsto \sigma\left(\frac{\tau(\beta)}{\beta}\right)^{\frac{1}{p^n}}\cdot \left(\frac{\sigma\tau(\beta)}{\beta}\right)^{-\frac{1}{p^n}}\cdot\left(\frac{\sigma(\beta)}{\beta}\right)^{\frac{1}{p^n}}$$ 
If $s:G_{S_{i,n}}\to G_{F_{i,n}}$ is a (continuous) section of the natural projection $G_{F_{i,n}}\to G_{S_{i,n}}$, this can be rewritten as
$$(\sigma,\,\tau) \mapsto s(\sigma)\left(\frac{s(\tau)(\beta^{\frac{1}{p^n}})}{\beta^{\frac{1}{p^n}}}\right)\cdot \left(\frac{s(\sigma\tau)(\beta^{-\frac{1}{p^n}})}{\beta^{-\frac{1}{p^n}}}\right)\cdot\left(\frac{s(\sigma)(\beta^{\frac{1}{p^n}})}{\beta^{\frac{1}{p^n}}}\right)=\frac{s(\sigma)s(\tau)(\beta^{\frac{1}{p^n}})}{s(\sigma\tau)(\beta^{\frac{1}{p^n}})}$$
But this is exactly a representative for $tg(f)$ (see the proof of Proposition 1.6.5 in \cite{NSW}). 
\end{proof}

The following lemma is stated in full generality but will be used in the proof of Theorem \ref{gh} only for the case when $p=2$ and $F$ is nonexceptional. 

\begin{lemma}\label{rabbercio}
Suppose that $p$ is odd or $F$ is nonexceptional. Then for any $i\geq 1$ and any $n\in\mathbb{N}$, $(Cl^{S_{i,n}}_{F_{i,n}})_p=0$ if and only if the map 
$$(\partial_{F_{i,n},i}^c)_{|_{H^2(F_{i,n},\,\mathbb{Z}_p(i+1))[p^n]}}:H^2(F_{i,n},\,\mathbb{Z}_p(i+1))[p^n]\longrightarrow \displaystyle{\bigoplus_{w\nmid p\infty}H^1((k_{i,n})_w,\,\mathbb{Z}_p(i))[p^n]}$$
is surjective.
\end{lemma}
\begin{proof}
Thanks to Lemma \ref{dic}, $$\mathrm{coker}\left((\partial_{F_{i,n},i}^{c})_{|_{H^2(F_{i,n},\,\mathbb{Z}_p(i+1))[p^n]}}\right)=0\Longleftrightarrow T_{F_{i,n},i,n}=0$$ 
where $T_{F_{i,n},i,n}$ is the map of Lemma \ref{dic}. Now the statement follows from Lemma \ref{maria} (note that the map $\pi_{i,n}$ in the statement of Lemma \ref{maria} is surjective). 
\end{proof}

\begin{remark}\label{trieste} 
An immediate consequence of the preceding lemma is the following: if $\mu_{p^n}\subseteq F$ for some \emph{positive} $n\in\mathbb{N}$ and $(Cl^{S}_F)_p\ne 0$, then the localization sequence for $K_{2i}(F)_p$ (or equivalently for $H^{2}(F,\,\mathbb{Z}_p(i+1))$) does not split for any $i\geq 1$.
\end{remark}

We are now ready to state and prove the main result of the paper. As the reader may guess, the case $p=2$ will require a different approach (and the results obtained are slightly weaker).  

\begin{teo}\label{gh}
Let $i$ be a positive integer and let $F$ be a number field.
\begin{itemize}
	\item If $p$ is odd, then for any $n\in \mathbb{N}$ there is an isomorphism 
	$$\Omega^{(p)}_{i,n}\cong \left(Cl^{S_{i,n}}_{F_{i,n}}\otimes \mu_{p^n}^{\otimes i}\right)_{\Gamma_{i,n}}$$ 
	In particular, the localization sequence for $K_{2i}(F)_p$ (or equivalently for $H^{2}(F,\,\mathbb{Z}_p(i+1))$) splits if and only if for every $n\in \mathbb{N}$ we have $\left(Cl^{S_{i,n}}_{F_{i,n}}\otimes \mu_{p^n}^{\otimes i}\right)_{\Gamma_{i,n}}=0$.
  \item If $p=2$ and $F$ is nonexceptional, then for $n\in \mathbb{N}$ sufficiently large there is an isomorphism 
  $$\Omega^{(2)}_{i,n}\cong \left(Cl^{S_{i,n}}_{F_{i,n}}\otimes \mu_{2^n}^{\otimes i}\right)_{\Gamma_{i,n}}$$ 
  Moreover the localization sequence for $K_{2i}(F)_2$ (or equivalently for $H^{2}(F,\,\mathbb{Z}_2(i+1))$) splits if and only if for every $n\in \mathbb{N}$ we have $\left(Cl^{S_{i,n}}_{F_{i,n}}\otimes \mu_{2^n}^{\otimes i}\right)_{\Gamma_{i,n}}=0$.  
\end{itemize}
\end{teo}
\begin{proof} 
For any $n\in\mathbb{N}$, there is a commutative diagram of $\Gamma_{i,n}$-modules (those on the bottom line have trivial action) with exact rows
$$
\begin{CD}
0@>>> H^2(G_{S_{i,n}},\,\mathbb{Z}_p(i+1)) @>>> H^{2}(F_{i,n},\,\mathbb{Z}_p(i+1))_p @>\widetilde{\partial}_i^c>> \displaystyle{\bigoplus_{w\nmid p}\,H^{1}((k_{i,n})_w,\,\mathbb{Z}_p(i))}@>>>0\\
@.@Vcor^{(2)}_nVV@Vcor^{(2)}_nVV@Vcor^{(1)}_nVV\\
0@>>> H^2(G_S,\,\mathbb{Z}_p(i+1)) @>>> H^{2}(F,\,\mathbb{Z}_p(i+1))_p @>\partial_i^c>> \displaystyle{\bigoplus_{v\nmid p}\,H^{1}(k_v,\,\mathbb{Z}_p(i))}@>>>0\\
\end{CD}
$$
where $cor_n$ is the appropriate (cohomological) corestriction. Following \cite{Hu}, Section 3, we consider a part of the commutative diagram induced by snake lemma, namely
$$
\begin{CD}
\displaystyle{\bigoplus_{w\nmid p}H^1((k_{i,n})_w,\,\mathbb{Z}_p(i))[p^n]}@>I^c_{F_{i,n},i,n}>> H^2(G_{S_{i,n}},\,\mathbb{Z}_p(i+1))/p^n\\
 @VVcor^{(1)}_n V   @VVcor^{(2)}_n V \\
\displaystyle{\bigoplus_{v\nmid p}H^1(k_v,\,\mathbb{Z}_p(i))[p^n]}@>I_{F,i,n}^c>> H^2(G_S,\,\mathbb{Z}_p(i+1))/p^n \\
\end{CD}
$$  
Using Lemma \ref{pn}, Lemma \ref{cf} and Lemma \ref{dic} we can rewrite it as follows
$$
\begin{CD}
\displaystyle{\bigoplus_{w\nmid p}\mu_{p^{n}}^{\otimes^{i}}}@>T_{F_{i,n},i,n}>> H^2(G_{S_n},\,\mu_{p^{n}}^{\otimes^{i+1}})\\
 @VVcor^{(0)}_n V   @VVcor^{(2)}_n V \\
\displaystyle{\bigoplus_{v\nmid p}H^0(k_v,\,\mu_{p^{n}}^{\otimes^{i}})}@>T_{F,i,n}>> H^2(G_S,\,\mu_{p^{n}}^{\otimes^{i+1}}) \\
\end{CD}
$$
Using Lemma \ref{maria}, we can complete the above diagram to get the following commutative diagram
$$
\begin{diagram}
\displaystyle{\bigoplus_{w\nmid p}\mu_{p^{n}}^{\otimes^{i}}}& &\rTo^{T_{F_{i,n},i,n}}& &H^2(G_{S_{i,n}},\,\mu_{p^{n}}^{\otimes^{i+1}})\\
\dTo&\rdTo^{\pi_{i,n}} & &\ruTo^{J_{i,n}}&\dTo \\
& &Cl^{S_{i,n}}_{F_{i,n}}\otimes \mu_{p^{n}}^{\otimes^{i}}\\
\displaystyle{\bigoplus_{v\nmid p}H^0(k_v,\,\mu_{p^{n}}^{\otimes^{i}})}& &\rTo^{T_{F,i,n}}& & H^2(G_S,\,\mu_{p^{n}}^{\otimes^{i+1}}) \\
\end{diagram}
$$
We remark explicitly that $\pi_{i,n}$ is surjective (since elements of the form $[\mathfrak{P}_w]\otimes\zeta_w$, with $\zeta_w$ running in $\mu_{p^n}^{\otimes i}$ generate $Cl^{S_{i,n}}_{F_{i,n}} \otimes \mu_{p^n}^{\otimes i}$). By definition of $\Omega^{(p)}_{i,n}$ and Lemma \ref{dic}, we have 
$$\mathrm{Im} \,T_{F,i,n}\cong \mathrm{Im} \,I^c_{F,i,n}=\Omega^{(p)}_{i,n}$$
Hence, $\eta_{i,n}=cor^{(2)}_n\circ J_{i,n}$ defines a $\Gamma_{i,n}$-homomorphism 
$$\eta_{i,n}:Cl^{S_{i,n}}_{F_{i,n}}\otimes \mu_{p^n}^{\otimes i}\to \Omega^{(p)}_{i,n}$$
which induces a map
$$\eta_{_{\Gamma_{i,n}}}:(Cl^{S_{i,n}}_{F_{i,n}}\otimes \mu_{p^n}^{\otimes i})_{\Gamma_{i,n}}\to \Omega^{(p)}_{i,n}$$
after taking coinvariants.\\
If $p$ is odd or $n$ is large enough, $\eta_{_{\Gamma_{i,n}}}$ is both injective (this follows from Proposition \ref{ke}) and surjective (in fact, $\eta_{i,n}$ is surjective because $\pi_{i,n}$ is surjective and $cor^{(1)}_n$ is surjective by Lemma \ref{surj}). In particular if $p$ is odd or $n$ is sufficiently large (and $F$ is nonexceptional)
$$\left(Cl^{S_n}_{F_n}\otimes \mu_{p^n}^{\otimes i}\right)_{\Gamma_n}\cong\Omega^{(p)}_{i,n}$$  
and the first statement and part of the second statement of the theorem are proved thanks to Lemma \ref{snake}.\\
We now focus on the case where $p=2$ (and of course $F$ is nonexceptional). Suppose first that the localization sequence for $H^{2}(F,\,\mathbb{Z}_2(i+1))$ splits, in other words $\Omega^{(2)}_{i,n}=0$ for any $n\in \mathbb{N}$. Observe that, since $\Gamma_{i,n}$ is a $2$-group, Nakayama's lemma implies 
\begin{equation}
(Cl^{S_{i,n}}_{F_{i,n}}\otimes \mu_{2^n}^{\otimes i})_{\Gamma_{i,n}}=0\Longleftrightarrow (Cl^{S_{i,n}}_{F_{i,n}})_2=0
\end{equation}
Moreover, if we set $F_{i,\infty}=\cup_{n\in\mathbb{N}}F_{i,n}$, then $F_{i,\infty}/F$ is a $\mathbb{Z}_2$-extension, since $F$ is nonexceptional. Now, by Lemma \ref{rabbercio}, we know that $(Cl_F^{S})_2=0$ (since $\{\pm1\}\subseteq F$). Moreover we also know that $(Cl^{S_{i,n}}_{F_{i,n}})_2=0$ for $n$ large enough, the map $\eta_{_{\Gamma_{i,n}}}$ being an isomorphism for such $n$. We have to show that $(Cl^{S_{i,n}}_{F_{i,n}})_2=0$ for any $n\in \mathbb{N}$. Let $m\in \mathbb{N}$ be such that $F_{i,m}\ne F$: we can find a prime $v$ above $2$ in $F$ which stays inert in $F_{i,m}/F$ because $(Cl^{S}_{F})_2=0$ and only primes above $2$ can ramify in $\mathbb{Z}_2$-extensions. In particular there is only one prime above $v$ in $F_{i,\infty}/F$ (since $F_{i,\infty}/F$ is a $\mathbb{Z}_2$-extension). Now take $n\geq m$ such that $(Cl^{S_{i,n}}_{F_{i,n}})_2=0$: $F_{i,n}/F_{i,m}$ has to be disjointed from the $2$-split Hilbert class field of $F_m$ since there is a nonsplit prime above $2$ in $F_{i,n}/F_{i,m}$. This means that the norm map 
$$(Cl^{S_{i,n}}_{F_{i,n}})_2\longrightarrow (Cl^{S_{i,m}}_{F_{i,m}})_2$$ 
has to be surjective, which implies $(Cl^{S_{i,m}}_{F_{i,m}})_2=0$.\\
Now suppose that $(Cl^{S_{i,n}}_{F_{i,n}})_2=0$ for every $n\in\mathbb{N}$: we prove by induction on $n$ that 
$\Omega^{(2)}_{i,n}=0$. Of course we trivially have $\Omega^{(2)}_{i,0}=0$ (and indeed also $\Omega^{(2)}_{i,1}=0$ by Lemma \ref{rabbercio}). Next consider the following commutative diagram
$$
\begin{CD}
H^2(F_{i,n},\,\mathbb{Z}_2(i+1))[2^n]@>\partial_{F_{i,n},i}^c>> \displaystyle{\bigoplus_{w\nmid p}H^1((k_{i,n})_w,\,\mathbb{Z}_2(i))[2^n]}\\
 @VVcor^{(2)}_n V   @VVcor^{(1)}_n V \\
H^2(F,\,\mathbb{Z}_2(i+1))[2^n]@>\partial_{F,i}^c>> \displaystyle{\bigoplus_{v\nmid p}H^1(k_v,\,\mathbb{Z}_2(i))[2^n]}\\
\end{CD}
$$   
A generic element $x=(x_v)\in\bigoplus_{v\nmid 2\infty}H^1(k_v,\,\mathbb{Z}_2(i))[2^n]$ can be written as a sum $x=y_1+\ldots+y_n$ where, for any $j=1,\,\ldots,\,n$, $y_j\in \bigoplus_{v\nmid 2\infty}H^1(k_v,\,\mathbb{Z}_2(i))[2^n]$ and
$$(y_j)_v=\left\{\begin{array}{ll}
x_v &\textrm{if $x_v\in H^1(k_v,\,\mathbb{Z}_2(i))[2^j]\smallsetminus H^1(k_v,\,\mathbb{Z}_2(i))[2^{j-1}]$}\\
0&\textrm{otherwise}
\end{array}\right.$$
By induction, for any $j=1,\,\ldots,\,n-1$, there is an element 
$$z_j\in H^2(F,\,\mathbb{Z}_2(i+1))[2^j]\subseteq H^2(F,\,\mathbb{Z}_2(i+1))[2^n]$$ 
such that $\partial_{F,i}^c(z_j)=y_j$. Now $(y_n)_v$ is an element of order $2^n$ in $H^1(k_v,\,\mathbb{Z}_2(i))[2^n]\cong H^0(k_v,\,\mu_{2^n}^{\otimes i})$ (by Lemma \ref{pn}). This means that $H^0(k_v,\,\mu_{2^n}^{\otimes i})=\mu_{2^n}^{\otimes i}$ which implies $(k_{i,n})_w=k_v(\mu_{2^n}^{\otimes i})=k_v$ for any $w|v$ in $F_{i,n}$ (by Lemma \ref{cf}), i.e. $v$ splits completely in $F_{i,n}/F$. In particular there is an element $w_n\in \bigoplus_{w\nmid 2\infty}H^1((k_{i,n})_w,\,\mathbb{Z}_2(i))[2^n]$ such that $cor_n^{(1)}(w_n)=y_n$. Now 
$$H^2(F_{i,n},\,\mathbb{Z}_2(i+1))[2^n]\stackrel{\partial_{F_{i,n},i}^c}{\longrightarrow}\bigoplus_{w\nmid p}H^1((k_{i,n})_w,\,\mathbb{Z}_2(i))[2^n]$$ 
is surjective by Lemma \ref{rabbercio} and therefore there exists an element $t_n\in H^2(F_{i,n},\,\mathbb{Z}_2(i+1))[2^n]$ such that $\partial_{F_{i,n},i}^c(t_n)=w_n$ and clearly $\partial_{F,i}^c(cor_n^{(2)}(t_n))=y_n$. This shows that
$$H^2(F,\,\mathbb{Z}_p(i+1))[2^n]\stackrel{\partial_{F,i}^c}{\longrightarrow} \displaystyle{\bigoplus_{v\nmid p}H^1(k_v,\,\mathbb{Z}_p(i))[2^n]}$$
is surjective which is equivalent to $\Omega^{(2)}_{i,n}=0$. 
\end{proof}

\begin{remark} It is worth noting (and I thank K. Hutchinson for pointing this out to me) that if $\mu_p\subseteq F$, then the splitting criterion of Theorem \ref{gh} is independent of $i$ (compare with Remark \ref{trieste}). In fact, in this situation, $\Gamma_n$ is a (cyclic) $p$-group and hence by Nakayama's lemma 
$$\left(Cl^{S_{i,n}}_{F_{i,n}}\otimes \mu_{p^n}^{\otimes^i}\right)_{\Gamma_{i,n}}=0 \Leftrightarrow Cl^{S_{i,n}}_{F_{i,n}}\otimes \mu_{p^n}^{\otimes^i}=0 \Leftrightarrow Cl^{S_{i,n}}_{F_{i,n}}/p^n=0 \Leftrightarrow (Cl^{S_{i,n}}_{F_{i,n}})_p=0$$
This holds in particular for $p=2$ and any nonexceptional number field $F$.
\end{remark}

\begin{remark} I don't know whether there is an isomorphism $\Omega^{(2)}_{i,n}\cong(Cl_{F_{i,n}}^{S_i,n}\otimes\mu_{2^n}^{\otimes i})_{\Gamma_{i,n}}$ for \emph{any} $n\in\mathbb{N}$.
\end{remark}

The criterion of Theorem \ref{gh} is closely related with the triviality of the so-called $i$-th \'etale wild kernel of $F$, which is denoted $W\!K^{\acute{e}t}_{2i}(F)$ and is isomorphic to $\mathrm{div}(K_{2i}(F)_p)$. Recall that, for any odd $p$ and any number field $E$, if, for any $v\mid p$, $E_v$ denotes a completion of $E$ at $v$, $W\!K^{\acute{e}t}_{2i}(E)$ is the kernel of the map
$$H^2(G_{E,\,S},\mathbb{Z}_p(i+1))\longrightarrow \bigoplus_{v\mid p}H^2_{\acute{e}t}(E_v,\,\mathbb{Z}_p(i+1))$$
which is induced by, for any $v\mid p$, an embedding of a fixed algebraic closure of $E$ in a fixed algebraic closure of $E_v$ (the kernel of the above map is actually independent of the chosen embedding).

\begin{teo}\label{TBKKSH}[Tate, Banaszak-Kolster, Keune, Schneider, \O stvaer]
Suppose that $p$ is an odd prime or $F$ is nonexceptional. For any $i\geq 1$, we have
$$\mathrm{div}(K_{2i}(F)_p)\cong W\!K^{\acute{e}t}_{2i}(F)\cong \lim_{\longleftarrow}\left(Cl^{S_{1,n}}_{F_{1,n}}\otimes\mu_{p^n}^{\otimes i}\right)_{\Gamma_n}$$
\end{teo}
\begin{proof}
The first isomorphism is due to Tate (for $i=1$ and $p$ odd, see \cite{Ta2}), Banaszak and Kolster (for any $i$ and $p$ odd, see \cite{Ba}, Theorem 3) and \O stvaer (for any $i$, $p=2$ and $F$ non exceptional, see \cite{Os}, Theorem 9.5).  The second isomorphism is due to Keune or Schneider (for $p$ odd, see \cite{Ke}, Theorem 6.6, or \cite{Sc}, §6, Lemma 1) and \O stvaer (for $p=2$ and $F$ nonexceptional, see \cite{Os}, Theorem 6.1). 
\end{proof}

\begin{remark}
As we have already remarked in the introduction of this paper, the condition $\mathrm{div}(K_{2i}(F)_p)=0$ is certainly necessary for the localization sequence for $K_{2i}(F)_p$ to split since 
$$\mathrm{div}(K_{2i}(\mathcal{O}_F)_p)=\mathrm{div}\left(\bigoplus_{v\nmid p\infty}K_{2i-1}(k_v)_p\right)=0$$ 
We can check that the criterion of Theorem \ref{gh} is indeed consistent with Theorem \ref{TBKKSH}, namely we can give another proof of the fact that, if $\left(Cl_{F_{i,n}}^{S_{i,n}}\otimes\mu_{p^n}^{\otimes i}\right)_{\Gamma_{i,n}}=0$ for any $n\in\mathbb{N}$, then $\mathrm{div}(K_{2i}(F)_p)=0$. This is not difficult thanks to the description given in Theorem \ref{TBKKSH}. For any $i\geq 1$, set 
$$F_{i,\infty}=\bigcup_{n\in \mathbb{N}}F_{i,n}\quad\textrm{and}\quad\Gamma_{i,\infty}=\mathrm{Gal}(F_{i,\infty}/F)$$
Next observe that $F_{i,\infty}\subseteq F_{1,\infty}$ and using the definition of $F_{i,n}$ one easily shows that the restriction gives an isomorphism 
$$\Gamma_{1,\infty}=\Gamma_{i,\infty}\times \Delta_{i,\infty}$$ 
where $\Delta_{i,\infty}=\mathrm{Gal}(F_{1,\infty}/F_{i,\infty})=\langle\gamma\in \Gamma_{1,\infty}\,|\, \gamma^i=1\rangle$. In particular $\Delta_{i,\infty}$ is finite of order coprime with $p$ (since $\Gamma_{1,n}\cong \mathbb{Z}_p\times \mathrm{Gal}(F(\mu_p)/F)$ and $\mathbb{Z}_p$ has no nontrivial finite subgroups) and acts trivially on $\mathbb{Z}_p(i)$. Now set 
$$X'_{i,\infty}=\lim_{\longleftarrow}Cl^{S_{i,n}}_{F_{i,n}}$$
and note that 
$$\left(X'_{i,\infty}\otimes\mathbb{Z}_p(i)\right)_{\Gamma_{i,\infty}}=\lim_{\longleftarrow}\left(Cl^{S_{i,n}}_{F_{i,n}}\otimes\mu_{p^n}^{\otimes i}\right)_{\Gamma_{i,n}}$$
We have $(X'_{1,\infty})_{\Delta_{i,\infty}}=X'_{i,\infty}$ since $\Delta_{i,\infty}$ has order coprime with $p$ and therefore
$$\mathrm{div}(K_{2i}(F)_p)\cong\left(X'_{1,\infty}\otimes\mathbb{Z}_p(i)\right)_{\Gamma_{1,\infty}}=\left(\left(X'_{1,\infty}\right)_{\Delta_{i,\infty}}\otimes\mathbb{Z}_p(i)\right)_{\Gamma_{i,\infty}}=\left(X'_{i,\infty}\otimes\mathbb{Z}_p(i)\right)_{\Gamma_{i,\infty}}$$
which shows that, if $\left(Cl_{F_{i,n}}^{S_{i,n}}\otimes\mu_{p^n}^{\otimes i}\right)_{\Gamma_{i,n}}=0$ for any $n\in\mathbb{N}$, then $\mathrm{div}(K_{2i}(F)_p)=0$. We will see in the next section (Example \ref{ce}) that the converse of this statement is not true in general: in other words, for any $i\geq 1$, there exists a prime number $p$ and a number field $F$ such that the localization sequence for $K_{2i}(F)_p$ does not split but $\mathrm{div}(K_{2i}(F)_p)=0$.
\end{remark}

\section{Examples.}\label{examples}
To begin with we analyze the simplest case, namely $F=\mathbb{Q}$. Let $p$ be an odd prime and let $A_{i,n}$ denote the $p$-Sylow subgroup of $Cl_{F_{i,n}}^{S_{i,n}}$ where $F_{i,n}=\mathbb{Q}(\mu_{p^n}^{\otimes i})$. Let $K_{1,n}$ be the $n$-th level of the cyclotomic $\mathbb{Z}_p$-extension of $\mathbb{Q}$. Set $\Delta_{1,n}=\mathrm{Gal}(F_{1,n}/K_n)$ and for every $j\in \mathbb{Z}$, let $A_{1,n}^{(j)}$ denote the $\omega^j$-component of $A_{1,n}$ where $\omega:\Delta_n\to \mathbb{Z}_p^\times$ denotes the Teichm\"uller character. 
We need the following well known result.

\begin{teo}[Kurihara]\label{trivial}
Let $p$ be an odd prime. For any $i\geq 1$, 
$$H^2(G_{\mathbb{Q},\,S},\mathbb{Z}_p(i+1))=0 \Longleftrightarrow A^{(-i)}=0$$
where $A$ is the $p$-Sylow of the class group of $\mathbb{Q}(\mu_p)$. In particular, the triviality of $H^2(G_{\mathbb{Q},\,S},\mathbb{Z}_p(i+1))$ only depends on the class of $i$ modulo $p-1$.
\end{teo}
\begin{proof}
See \cite{Ku}, Corollary 1.5.
\end{proof}


\begin{example}\label{qe}
Let $F$ be the field of rational numbers. Since $\mathrm{Gal}(F_{1,n}/F_{i,n})$ acts trivially on $\mu_{p^n}^{\otimes i}$, for every $n\geq 1$ we have 
$$\left(A_{1,n}\otimes \mu_{p^n}^{\otimes^i}\right)_{\Gamma_{1,n}}=\left(\left(A_{1,n}\right)_{\mathrm{Gal}(F_{1,n}/F_{i,n})}\otimes\mu_{p^n}^{\otimes^i}\right)_{\Gamma_{i,n}}=\left(A_{i,n}\otimes\mu_{p^n}^{\otimes^i}\right)_{\Gamma_{i,n}}$$ 
the last equality coming from the fact that there is only one ramified prime in $F_{1,n}/\mathbb{Q}$ and it is totally ramified (use \cite{Wa}, Proposition 13.22). By Nakayama's lemma 
$$\left(\left(A_{1,n}\otimes\mu_{p^n}^{\otimes^i}\right)_{\Delta_{1,n}}\right)_{\mathrm{Gal}(K_{1,n}/\mathbb{Q})}=0\Longleftrightarrow\left(A_{1,n}\otimes \mu_{p^n}^{\otimes^i}\right)_{\Delta_{1,n}}=0$$
Furthermore 
$$\left(A_{1,n}\otimes \mu_{p^n}^{\otimes^i}\right)_{\Delta_{1,n}}\cong A_{1,n}^{(-i)}$$
Moreover, it is easy to see that for any $n\geq 1$
$$A_{1,n}^{(-i)}=0 \Longleftrightarrow A_{1,1}^{(-i)}=0$$
Hence the localization sequence for $K_{2i}(\mathbb{Q})_p$ splits if and only if $A_{1,1}^{(-i)}$ is trivial. Therefore, by Theorem \ref{trivial}, the localization sequence for $K_{2i}(\mathbb{Q})_p$ splits if and only if $H^2(G_{\mathbb{Q},\,S},\mathbb{Z}_p(i+1))=0$ (the latter is equivalent to $K_{2i}(\mathbb{Z})_p=0$ under the Quillen-Lichtenbaum conjecture). Moreover, it can be easily proved that $W\!K_{2i}^{\acute{e}t}(\mathbb{Q})\cong H^2_{\acute{e}t}(G_{\mathbb{Q},\,S},\mathbb{Z}_p(i+1))$: therefore in this case the triviality of $\mathrm{div}(K_{2i}(\mathbb{Q})_p)$ is a necessary and sufficient condition for the localization sequence for $K_{2i}(\mathbb{Q})_p$ to split. 
\end{example}

\noindent
Anyway in general the condition $W\!K^{\acute{e}t}_{2i}(F)=0$ is weaker than the condition of Theorem \ref{gh}, as we will show in the next example. First we need the following criterion.

\begin{prop}\label{jau}
Let $p$ be an odd prime and let $F/\mathbb{Q}$ be finite Galois extension such that 
\begin{itemize}
	\item $\mu_{p}\subseteq F$;
	\item $(Cl^{S}_F)_p\cong \mathbb{Z}/p\mathbb{Z}$;
	\item there is only one prime above $p$ in $F$;
	\item $F(\mu_{p^2})/F$ is a nontrivial extension where every prime over $p$ is totally split.
\end{itemize}
Then, for any $i\geq 1$, $W\!K^{\acute{e}t}_{2i}(F)_p$ is trivial but the localization sequence for $K_{2i}(F)_p$ does not split.  
\end{prop}
\begin{proof}
We are going to use Jaulent's theory of logarithmic classes: for notation and basic results the reader is referred to \cite{Ja}. For any $F/\mathbb{Q}$ finite and Galois (even not satisfying the hypotheses) we have an exact sequence (see \cite{DS}, §3) 
\begin{equation}\label{seds}
0 \to \widetilde{Cl}_{F}(p)\to \widetilde{Cl}_{F} \stackrel{\varphi}{\longrightarrow} (Cl^{S}_F)_p\longrightarrow \mathrm{deg}_F \mathcal{D}\ell/(\mathrm{deg}_{F}\mathfrak{p})\mathbb{Z}_p\to 0
\end{equation} 
where $\mathfrak{p}$ is any prime of $F$ over $p$. Moreover 
$$
\mathrm{deg}_F \mathcal{D}\ell=p[F\cap\widehat{\mathbb{Q}}^c:\,\mathbb{Q}]\mathbb{Z}_p\quad\textrm{and}\quad\mathrm{deg}_{F}\mathfrak{p}=\tilde{f}_{\mathfrak{p}}\cdot \mathrm{deg}\, p=[F_{\mathfrak{p}}\cap\widehat{\mathbb{Q}}^c_p:\,\mathbb{Q}_p]\cdot p
$$
where $\widehat{\mathbb{Q}}^{c}$ (resp. $\widehat{\mathbb{Q}}^c_p$) is the cyclotomic $\widehat{\mathbb{Z}}$-extension of $\mathbb{Q}$ (resp. $\mathbb{Q}_p$). Now we want to compare $[F\cap\widehat{\mathbb{Q}}^c:\,\mathbb{Q}]$ and $[F_{\mathfrak{p}}\cap\widehat{\mathbb{Q}}^c_p:\,\mathbb{Q}_p]$. We have $v_p([F\cap\widehat{\mathbb{Q}}^c:\,\mathbb{Q}])=0$ thanks to the first and the fourth hypothesis. The fourth hypothesis also implies that $v_p([F_{\mathfrak{p}}\cap\widehat{\mathbb{Q}}^c_p:\,\mathbb{Q}_p])=s\geq 1$.  
In other words 
$$\mathrm{deg}_F \mathcal{D}\ell/(\mathrm{deg}_{F}\mathfrak{p})\mathbb{Z}_{p}\cong \mathbb{Z}/p^s\mathbb{Z}$$ 
with $s\geq 1$. Now the third hypothesis implies $\widetilde{Cl}(p)=0$ (see \cite{DS}, Lemma 4): therefore (\ref{seds}) implies that $s=1$ and $\widetilde{Cl}_F=0$ because of the second hypothesis. Since $\mu_p\subseteq F$, we can use the isomorphism (see \cite{JM}, Théorème 3) 
$$\mu_{p}^{\otimes i}\otimes\widetilde{Cl}_F\cong W\!K^{\acute{e}t}_{2i}(F)/pW\!K^{\acute{e}t}_{2i}(F)$$
to deduce that $W\!K_{2i}^{\acute{e}t}(F)=0$. On the other hand $(Cl^{S}_F)_p$ is nontrivial, hence Theorem \ref{gh} (or already Remark \ref{trieste}) tells us that the localization sequence for $K_{2i}(F)_p$ does not split.
\end{proof}


\noindent
\begin{example}\label{ce} 
We have to find a number field $F$ and an odd prime $p$ satisfying the hypotheses of Proposition \ref{jau}. We proceed as follows: we choose an odd prime $p$ and a prime $\ell$ such that $\ell\equiv 1\,\,(mod \,\,p)$: this ensures that $\mathbb{Q}(\zeta_\ell)$ has exactly one subextension of degree $p$ which we call $E$. Let $K$ be the subextension of degree $p$ of $\mathbb{Q}(\mu_{p^2})$: then $EK$ is an abelian number field whose Galois group is isomorphic to $(\mathbb{Z}/p\mathbb{Z})^2$. Let $F'$ be a subextension of degree $p$ of $EK$ which is different from $E$ and $K$. Now, if the order of $p$ modulo $\ell$ is not divisible by $p$, then $E$ has to be totally split at $3$. In particular, $EK/F'$ is totally split at $p$ and $F'/\mathbb{Q}$ has only one prime above $p$ (which is totally ramified). We may then choose $F=F'(\mu_p)$: then the first, the third and the fourth hypotheses of Proposition \ref{jau} are satisfied. So we are left to find such a prime $\ell$ with the additional requirement that $(Cl^{S}_F)_p$ is cyclic of order $p$.\\ 
Choose $p=3$ and $\ell=61$: of course we have $61\equiv 1\,\,(mod\,\, 3)$ and $3$ has order $10$ modulo $61$. We only have one choice for $F'$ and computations with PARI (\cite{PA}) reveal that $F=F'(\mu_3)=\mathbb{Q}(\theta)$ where $\theta$ is a root of the polynomial $X^6-793X^3+226981$. We clearly have only one (totally ramified) prime above $3$ in $F$ and furthermore $(Cl^{S}_F)_3\cong \mathbb{Z}/3\mathbb{Z}$. Then by Proposition \ref{jau}, for any $i\geq1$, we deduce that $W\!K^{\acute{e}t}_ {2i}(F)_3=0$ but the localization sequence for $K_{2i}(F)_3$ does not split. 
\end{example}

\begin{remark} (i) It seems reasonable to conjecture that, for any $i\geq 1$ and any rational prime $p$, there exist infinitely many number fields $F$ such that the localization sequence for $K_{2i}(F)_p$ does not split but $\mathrm{div}(K_{2i}(F)_p)=0$.\\
	(ii) As we have seen there exist a number field $F$ and a prime $p$ such that (for any $i\geq 1$) $\mathrm{div}(K_{2i}(F)_p)=0$ but the localization sequence for $K_{2i}(F)_p$ does not split. However, for any number field $F$, any prime $p$ and any $i\geq 1$, we have $\mathrm{div}(K_{2i}(F)_p)=0$ if and only if $K_{2i}(F)_p$ is isomorphic to a direct sum of finite cyclic groups. This follows from a theorem of Pr\"ufer (see \cite{Ka}, Theorem 11).
\end{remark}

\noindent
\textsc{Luca Caputo}\\
School of Mathematical Sciences\\
University College Dublin\\
Belfield, Dublin 4, Ireland\\
luca.caputo@ucd.ie

\end{document}